\documentclass[11pt,reqno]{amsart}
\usepackage[utf8]{inputenc}
\usepackage{setspace}

\usepackage{amsmath,amssymb,amsthm,mathrsfs,graphicx,caption,booktabs,multirow,adjustbox,makecell}
\usepackage[margin=3.4cm]{geometry}
\usepackage{amsthm}

\newtheoremstyle{tight}
{0pt}   
{0pt}   
{}      
{}      
{\bfseries} 
{.}     
{0.5em} 
{}      

\theoremstyle{tight}

\newtheorem{theorem}{Theorem}[section]
\newtheorem{example}{Example}[section]
\newtheorem{definition}{Definition}[section]
\newtheorem{corollary}{Corollary}[section]
\newtheorem{lemma}{Lemma}[section]
\newtheorem{remark}{Remark}[section]
\newtheorem{property}{Property}[section]
\title[Cumulative Residual Mathai--Haubold Entropy]{Nonparametric Inference for Cumulative Residual Mathai--Haubold Entropy of order $\alpha$ }

\author[]{A\lowercase{nija} C.R.\lowercase{\textsuperscript{a}} , S\lowercase{mitha} S.\lowercase{\textsuperscript{a}} \lowercase{and} S\lowercase{udheesh} K. K\lowercase{attumannil.\textsuperscript{b}}
\\
\lowercase{\textsuperscript{a}}K\lowercase{uriakose} E\lowercase{lias} C\lowercase{ollege},
  M\lowercase{annanam},
  K\lowercase{erala},
  I\lowercase{ndia,}\\
\lowercase{\textsuperscript{b}}I\lowercase{ndian} S\lowercase{tatistical} I\lowercase{nstitute},
  C\lowercase{hennai}, I\lowercase{ndia.}}
  \thanks{Corresponding author email: skkattu@isichennai.res.in}
\begin{document}
\maketitle
\doublespacing
\vspace{-0.2in}

\begin{abstract}
 In this paper, we study the properties of cumulative residual Mathai--Haubold entropy of order $\alpha$. A dynamic version of this entropy measure is then proposed, and its properties are examined within the framework of reliability modeling. We show that the dynamic cumulative residual Mathai--Haubold entropy of order $\alpha$ uniquely determines the survival function. Characterization results for the exponential and generalized Pareto distributions are derived using the proposed measure. Furthermore, we develop nonparametric estimators for the cumulative residual Mathai--Haubold entropy and its dynamic counterpart of order $\alpha$, based on the kernel estimation of the survival function. The performance of these estimators is evaluated through a Monte Carlo simulation study. Finally, the practical relevance of the proposed dynamic estimator is illustrated using two real data; failure-time data from aircraft windshields and failure-time data from 40 randomly selected mechanical switches.
\end{abstract}
\doublespacing
\section{Introduction}
Information measures play a central role in probability theory, reliability analysis, information theory, signal processing, machine learning, economics, physics, and biological sciences by quantifying uncertainty associated with random variables and lifetime distributions.
Measuring the uncertainty of a random variable is a fundamental problem in information theory. In this context, Shannon (1948) introduced entropy, which lays the foundation of information theory and quantifies the average uncertainty or information content associated with the outcomes of a random variable.

 Let $X$ be a non-negative random variable with probability density function $f(x)$. The Shannon entropy of $X$ is defined as
\begin{align}\label{entropy}
H(X) = -\int_{0}^{\infty} f(x) \log f(x) \, dx = E[-\log f(X)],
\end{align}
where $\log$ denotes the natural logarithm.
\par Several generalizations of Shannon entropy have been proposed in the literature through the introduction of additional parameters, making these measures more sensitive to different shapes of probability distributions. In this context, Mathai and Haubold (2006) introduced a generalized information measure, known as the Mathai--Haubold entropy (MHE), defined as
\begin{align}\label{MHE}
M_{\alpha}(X) = \frac{1}{\alpha-1}\left( \int_{0}^{\infty} f^{2-\alpha}(x)\,dx - 1 \right),
\qquad \alpha \neq 1,\; 0<\alpha<2.
\end{align}

 The above measure can also be written as
 \begin{align}\label{reMHE}
M_\alpha(X)= \frac{1}{\alpha-1}\left({\int_0^\infty} f^{2-\alpha}(x)\,dx  - {\int_0^\infty}f(x) \, dx\right), \alpha \neq 1, 0<\alpha<2.
 \end{align}
The generalizing parameter $\alpha$ provides greater flexibility in quantifying uncertainty, making the MHE a more general measure than Shannon’s entropy.
 Further, when $\alpha \to 1$, $M_\alpha(X)$  reduces to Shannon’s entropy given in (\ref{entropy}).

 \par The uncertainty associated with the remaining lifetime of a component can be quantified using residual entropy, introduced by Ebrahimi (1996). It is defined as
\begin{align}\label{residual entropy}
H(X;t) = -\int_{t}^{\infty} \frac{f(x)}{\bar{F}(t)}
\log\!\left( \frac{f(x)}{\bar{F}(t)} \right) dx,
\end{align}
where $\bar{F}(t) = 1 - F(t)$ represents the survival function. The residual entropy function has applications in lifetime data analysis, reliability modeling, and the study of survival and aging properties.

\par As an extension of the Mathai--Haubold entropy to age-dependent settings, Dar and Al--Zahrani (2013) introduced the Mathai--Haubold residual entropy to model uncertainty in the remaining lifetime of a system, which is defined as
\begin{align}
 M_\alpha(X;t)= \frac{1}{\alpha-1}\left( {\int_t^\infty} \left(\frac{{f}(x)}{\bar{F}(t)}\right)^{2-\alpha} \,dx -1\right), \alpha \neq 1, 0<\alpha<2,t>0.
\end{align}
\par Several information measures that were originally defined in terms of probability density functions have subsequently been reformulated using the survival function. This shift is motivated by the fact that the survival function offers a more intuitive and practically useful representation of system lifetime behavior, particularly in the presence of censoring or truncation. In this context, Rao et al. (2004) introduced an alternative entropy measure based on the survival function, referred to as the Cumulative Residual Entropy (CRE) and is given by
\begin{align}\label{CRE}
     CRE(X)=-\int_{0}^{\infty} \bar{F}(x) \log \bar{F}(x) \,dx.
\end{align}

  The CRE is always non-negative and has practical applications in fields such as reliability analysis and image alignment.
 \par Asadi and Zohrevand (2007) proposed a modification of the cumulative residual entropy (CRE), referred to as the dynamic cumulative residual entropy and is defined as
 \begin{align}\label{CREt}
 CRE(X;t)=-\int_{t}^{\infty} \frac{\bar{F}(x)} {\bar{F}(t)} \log \left( \frac{\bar{F}(x)}{\bar{F}(t)} \right) dx.
 \end{align}
  \par Smitha et al. (2023) proposed an alternative definition based on the survival function of the entropy generating function introduced by Golomb (1966) namely cumulative residual entropy generating function (CREGF). They also proposed its dynamic version called dynamic cumulative residual entropy generating function (DCREGF) and is given by\\
  \begin{align}\label{C(s)}
  C_s{(X;t)}=\int_t^\infty\left(\frac{\bar{F}(x)}{\bar{F}(t)}\right)^s dx, s>0.
 \end{align}

 Sati and Gupta (2015) proposed the cumulative residual Tsallis entropy and its dynamic version as extensions of the Tsallis entropy (Tsallis, 1988). These measures are defined respectively by
\begin{align}\label{eta}
\eta_{\alpha}(X)
= \frac{1}{\alpha-1}
\left[
1-\int_{0}^{\infty} (\bar{F}(x))^{\alpha}\,dx
\right],
\qquad \alpha>0,\; \alpha\neq 1,
\end{align}
and
\begin{align}\label{eta t}
\eta_{\alpha}(X;t)
= \frac{1}{\alpha-1}
\left[
1-\int_{t}^{\infty}
\left(
\frac{\bar{F}(x)}{\bar{F}(t)}
\right)^{\alpha}
\,dx
\right],
\qquad \alpha>0,\; \alpha\neq 1.
\end{align}

Subsequently, Rajesh and Sunoj (2016) introduced alternative forms of the cumulative residual Tsallis entropy and its dynamic counterpart corresponding to \eqref{eta} and \eqref{eta t}. These measures are given by
\begin{align}\label{a eta}
\psi_{\alpha}(X)
= \frac{1}{\alpha-1}
\int_{0}^{\infty}
\left(
\bar{F}(x)-(\bar{F}(x))^{\alpha}
\right)dx,
\qquad \alpha>0,\; \alpha\neq 1,
\end{align}
and
\begin{align}\label{a eta t}
\psi_{\alpha}(X;t)
= \frac{1}{\alpha-1}
\int_{t}^{\infty}
\left(
\frac{\bar{F}(x)}{\bar{F}(t)}
-
\left(
\frac{\bar{F}(x)}{\bar{F}(t)}
\right)^{\alpha}
\right)dx,
\qquad \alpha>0,\; \alpha\neq 1.
\end{align}

 Recently, Anija et al. (2025) introduced and extensively studied the cumulative residual Mathai--Haubold entropy and its dynamic counterpart respectively as
\begin{align}\label{CRM}
\mathrm{CRM}_{\alpha}(X)
&=
\frac{1}{\alpha-1}
\left(
\int_{0}^{\infty}
(\bar{F}(x))^{\,2-\alpha}\,dx
-1
\right),
\qquad
\alpha \neq 1,\; 0<\alpha<2,
\end{align}
and
\begin{align}\label{CRMt}
\mathrm{CRM}_{\alpha}(X_t)
&=
\frac{1}{\alpha-1}
\left(
\int_{t}^{\infty}
\left(
\frac{\bar{F}(x)}{\bar{F}(t)}
\right)^{\,2-\alpha}
\,dx
-1
\right),
\qquad
\alpha \neq 1,\; 0<\alpha<2,\; t>0.
\end{align}


Recently, Alnssyan and Dar (2026)  developed and analyzes a Mathai–Haubold entropy framework based on cumulative distributions and order-statistic representations.

In the present paper, we
investigate several additional properties and applications that further demonstrate its relevance in reliability theory and lifetime data analysis. We also introduce and study its residual counterpart.
The paper is organized as follows. In Section 2, we studies the properties of  cumulative residual Mathai--Haubold entropy of order $\alpha$, while Section 3 presents the dynamic cumulative residual Mathai–Haubold entropy  of order $\alpha$ , discusses its properties and shows that it uniquely determines the distribution. In Section 4, we characterizes exponential distribution and GPD  using the proposed measure while Section 5 deals with new classes of lifetime distributions and introduces a hazard rate ordering based on the DCRMHE of $\alpha$. In section 6, we focus on the nonparametric kernel estimation of CRMHE and DCRMHE of order $\alpha$. For analyzing the performance of the proposed estimators, Section 7 presents Monte Carlo simulation studies. Section 8 deals with a real data application of failure times of aircraft windshields. Finally, Section 9 provides the summary and outlook.
\par
\section{Cumulative residual  Mathai--Haubold  entropy of order $\alpha$ \((\zeta_{\alpha}(X)\))
}

We note that Alnssyan and Dar (2026) introduced the Cumulative Mathai--Haubold Entropy (CMHRE), which is the same measure referred to herein as the cumulative residual Mathai--Haubold entropy (CRMHE) of order $\alpha$, as given in Definition 2.1. In this section, we investigate in detail several fundamental properties of this measure and establish results that highlight its applicability in the analysis of lifetime distributions.

 \begin{definition}
 For a continuous non-negative random variable $X$ with survival function $\bar{F}(x)$, cumulative residual Mathai--Haubold entropy of order $\alpha$ is denoted by $\zeta_{\alpha}(X)$, and is defined as
\begin{equation}\label{ACRM}
\begin{split}
\zeta_{\alpha}(X)
&= \frac{1}{\alpha-1}\left(
    \int_0^\infty (\bar{F}(x))^{2-\alpha}\,dx
    - \int_0^\infty \bar{F}(x)\,dx
\right) \\
&= \frac{1}{\alpha-1}\left(
    \int_0^\infty (\bar{F}(x))^{2-\alpha}\,dx
    - \mu
\right), 0<\alpha<2, \alpha\neq 1.
\end{split}
\end{equation}

\end{definition}
where
\[
\mu=\int_0^\infty \overline F(x)\,\mathrm{d}x.
\]
It is to be noted that $\zeta_{\alpha}(X)\geq0,for ~ 0<\alpha<2.$\\
 The mean residual life (MRL) function of a non-negative random variable  $X$ with survival function $\overline F(x)$ is given by
\[
m(t)= {E}[X-t\mid X>t]
=\frac{\int_t^\infty \overline F(u)\,\mathrm{d}u}{\overline F(t)},
\]
whenever $\overline F(t)>0$.\\
Also note that
 \[
\left. \zeta^{'}_{\alpha}(X)\right|_{\alpha=0}=\int_0^\infty (\bar F(x))^2 \log \bar F(x) \,\mathrm{d}x,
\]
where, the prime denote the derivative.\\
 Next property shows the relationship of $\zeta_{\alpha}(X)$ with the mean residual life function.
\begin{property}
From (\ref{ACRM}),we can observe that
\[
\begin{split}
\zeta_{\alpha}(X)
&=\frac{1}{1-\alpha}\Biggl(\mu
-\int_{0}^{\infty}(\bar{F}(x))^{2-\alpha}\,\mathrm{d}x\Biggr)\\
&=\frac{1}{1-\alpha}\Biggl(\mu
+\int_{0}^{\infty}
\left(\frac{\mathrm{d}}{\mathrm{d}x}\big(m(x)\,\bar{F}(x)\big)\right)
(\bar{F}(x))^{1-\alpha}\,\mathrm{d}x\Biggr)\\
&= {E} \!\big[m(X)\,(\bar{F}(X))^{1-\alpha}\big].
\end{split}
\]

\end{property}
\begin{example}
Suppose $X$ follows exponential distribution with survival function $\bar{F}(x)=e^{-\lambda x},x>0,\lambda>0$, then $\zeta_{\alpha}(X)=\frac{1}{\lambda(2-\alpha)}$, $m(x)=\frac{1}{\lambda}$ and $E(m(X)~(\bar{F}(X))^{1-\alpha})=\frac{1}{\lambda(2-\alpha)}.$
\end{example}
\begin{example}
    Suppose X have uniform distribution with survival function $\bar{F}(x)=1-\frac{x}{a}, 0<x<a,a>0$, then $\zeta_{\alpha}(X)= \frac{a}{2(3-\alpha)}$, $m(x)=\frac{a-x}{2}$ and $E(m(X)~(\bar{F}(X))^{1-\alpha})=\frac{a}{2(3-\alpha)}.$
\end{example}
\begin{corollary}
    Even if (\ref{MHE}) and (\ref{reMHE}) are equal, the corresponding cumulative Mathai--Haubold entropy measures given in (\ref{CRM}) and  (\ref{ACRM}) are not same as $\zeta_{\alpha}(X)= CRM_{\alpha}(X)+\frac{1-\mu}{\alpha-1}$.
\end{corollary}
The following property shows that $\zeta_{\alpha}(X)$ is a shift-independent measure.
\begin{property}
 Let $X$ be continuous non negative random variable and  $Y=a X+b$ with $a>0$ and $b \geq 0$ then,
$$
\zeta_{\alpha}(Y)=a~\zeta_{\alpha}(X).
$$
\end{property}

\noindent{\bf Proof:}  The proof follows using the fact that $\bar{F}_{a X+b}(x)=\bar{F}_{X}\left(\frac{x-b}{a}\right)$ for all $x>b$.


Next we obtain bounds for $\zeta_{\alpha}(X)$ in terms of CRE(X).
\begin{property}
    Let $X$ be a non negative continuous random variable with survival function $\bar{F}(x)$ then, $\zeta_{\alpha}(X)\leq(\geq)CRE(X)$  if $o<\alpha<1(1<\alpha<2)$, where CRE(X) is given in (\ref{CRE}).
\end{property}
\begin{proof}
For $o<\alpha<1(1<\alpha<2)$, we have\\

 \begin{align}
\zeta_{\alpha}(X)
&= \frac{1}{\alpha-1}
\int_0^\infty
\left( (\bar{F}(x))^{2-\alpha}
- \bar{F}(x) \right)\,dx \nonumber\\
&= \frac{1}{1-\alpha}
\int_0^\infty
\bar{F}(x)
\left( 1 - (\bar{F}(x))^{1-\alpha} \right)\,dx \nonumber\\
&\leq(\geq)-\int_{0}^{\infty}
\bar{F}(x)\log \bar{F}(x)\,dx \nonumber\\
&= \mathrm{CRE}(X).
\end{align}

\end{proof}
\begin{corollary}
From the above property, it is clear that, if $a=1$
\[
\zeta_{\alpha}(Y)=\zeta_{\alpha}(X).
\]
\end{corollary}
\begin{property}
The cumulative residual Mathai--Haubold entropy of order $\alpha$, $\zeta_{\alpha}(X)$ of a non-negative random variable $X$ can be expressed in terms of the cumulative residual Mathai--Haubold entropy of its equilibrium random variable $X_E$.
\end{property}

\begin{proof}
 Let $X$ be a non-negative random variable with density function $f(x)$, and let $X_E$ be the equilibrium random variable corresponding to $X$ with density function
\[
g(x)=\frac{\bar{F}(x)}{E(X)}, \qquad x \ge 0,
\]
where $\bar{F}(x)=1-F(x)$ and $\mu=E(X)$.
Using~(\ref{reMHE}), we have
\[
(\alpha-1) M_{\alpha}(X_E)
= \int_{0}^{\infty}
\left( \frac{\bar{F}(x)}{\mu} \right)^{2-\alpha} \, dx
- \int_{0}^{\infty}
\frac{\bar{F}(x)}{\mu} \, dx .
\]
Since
\[
\int_{0}^{\infty} \bar{F}(x)\,dx=\mu,
\]
the above expression simplifies to
\[
(\alpha-1) M_{\alpha}(X_E)
= \int_{0}^{\infty}
\left( \frac{\bar{F}(x)}{\mu} \right)^{2-\alpha} \, dx - 1.
\]
Consequently, the relationship between $M_{\alpha}(X)$ and $M_{\alpha}(X_E)$ is given by
\[
(\alpha-1) M_{\alpha}(X)
= \mu^{\,2-\alpha}
\bigl[(\alpha-1) M_{\alpha}(X_E) - 1\bigr]
- \mu .
\]

\end{proof}


 Under the Proportional hazards  model assumption, the survival functions of the random variables $X$ and $X_{\theta}^*$ satisfy the following relationship
  $$\bar{F}_\theta^*(x)=(\bar{F}(x))^\theta, \theta>0,x \in R.$$
   The following property establishes a relationship between $\zeta_{\alpha}(X)$ and cumulative Tsallis entropy of order $\alpha$ proposed by Rajesh and Sunoj (2016).
\begin{property}
$(\alpha-1)~\zeta_{\alpha}(X_\theta^*)=(\beta-1)~\zeta_{\beta}(X)+(\theta-1)~\psi_{\theta}(X)$,
where $\beta=2-\theta(2-\alpha)$ and $\psi_{\theta}(X)$ is the  cumulative
residual Tsallis entropy of order $\theta$ given in (\ref{a eta}).
\end{property}
\begin{proof}
   We have
  \[ (\alpha-1)~\zeta_{\alpha}(X_\theta^*)= {\int_0^\infty} (\bar{F}(x))^{\theta(2-\alpha)} \,dx -{\int_0^\infty} (\bar{F}(x))^{\theta} \,dx
  \]
\[
\qquad\qquad\qquad\qquad=({\beta-1})\zeta_{\beta}(X)+({\theta-1})\psi_{\theta}(X), where~ \beta=2-\theta(2-\alpha),
\]
as required.
\end{proof}
\begin{example}
Consider a series system consisting of $n$ components with independent and identically distributed lifetimes $X_i$, $i=1,2,\ldots,n$, having common distribution function $F$. Let $X_{1:n}$ denote the first order statistic based on the random sample $X_1,X_2,\ldots,X_n$ from $F$. Then the survival function of $X_{1:n}$ is given by
\[
\bar{F}_{X_{1:n}}(x)=\bigl(\bar{F}(x)\bigr)^n.
\]
Observe that the component lifetime $X_i$, $i=1,2,\ldots,n$, and the system lifetime $X_{1:n}$ satisfy the proportional hazards model. Further, in view of the cumulative residual Mathai--Haubold entropy of order $\alpha$, we have
\[
(\alpha-1)\zeta_{\alpha}(X_{1:n})
= (\beta-1)\zeta_{\beta}(X) + (n-1)\psi_{n}(X),
\qquad \text{where } \beta = 2-n(2-\alpha).
\]
Hence, the above relation follows as claimed in Property~2.5.
\end{example}

\begin{example}
Suppose the random variable
$X$ follows exponential distribution with parameter $\lambda$ then


$(\alpha-1)\zeta_\alpha(X_\theta^*)=\frac{1}{\lambda\theta(2-\alpha)}-\frac{1}{\lambda \theta}$, $(\theta-1)\psi_{\theta}(X)=\frac{1}{\lambda}-\frac{1}{\lambda \theta}$ and $(\beta-1)\zeta_{\beta}(X)=\frac{1}{\lambda(2-\beta)}-\frac{1}{\lambda}$, where $2-\beta=\theta(2-\alpha)$. So we have $(\alpha-1)~\zeta_{\alpha}(X_\theta^*)=(\beta-1)~\zeta_{\beta}(X)+(\theta-1)~\psi_{\theta}(X)$.
\end{example}




\section{Dynamic cumulative residual Mathai--Haubold entropy(DCRMHE) of order $\alpha$}
The residual lifetime of a component or a system is a fundamental concept in reliability theory and survival analysis.
In reliability and life-testing experiments, the observed data are often truncated, and in such situations,
information measures defined for the entire lifetime distribution may not adequately capture the uncertainty associated
with the remaining life. In particular, the cumulative residual Mathai--Haubold entropy is not suitable
when the focus is on the post-survival behaviour of a component.

Motivated by this practical consideration, we introduce the dynamic cumulative residual Mathai--Haubold
entropy of order $\alpha$, which quantifies the uncertainty associated with the residual lifetime distribution
and provides a natural extension of the cumulative residual Mathai--Haubold  framework for reliability analysis under truncation.

\begin{definition}
Consider the lifetime of a component or system as $X$, and let it survive up to time $t$. In such cases, we consider the random variable $X_t=X-t|X>t$, which is time-dependent or dynamic with the survival function given by

\[
\bar{F}_t(x)=
\begin{cases}
    \frac{\bar{F}(x)}{\bar{F}(t)}, & \text{if $x>t$},\\
    1, & \text{$otherwise$}.
\end{cases}
\]
\end{definition}

\begin{definition}
 For a non-negative continuous random variable $X$ with survival function $\bar{F}(x)$, DCRMHE of order $\alpha$, is denoted by $\zeta_\alpha(X_t)$ is defined as
\begin{align}\label{ACRM t}
    \zeta_\alpha(X_t)= \frac{1}{\alpha-1}\left( {\int_t^\infty} \left(\frac{\bar{F}(x)}{\bar{F}(t)}\right)^{2-\alpha} \,dx -{\int_t^\infty} \frac{\bar{F}(x)}{\bar{F}(t)} \,dx\right), \alpha \neq 1, 0<\alpha<2.
\end{align}
$\zeta_\alpha(X_t)$ can be expressed in terms of mean residual life function $m(t)$ as
\begin{align}
\zeta_\alpha(X_t)= \frac{1}{\alpha-1}\left( {\int_t^\infty} \left(\frac{\bar{F}(x)}{\bar{F}(t)}\right)^{2-\alpha} \,dx -m(t)\right), \alpha \neq 1, 0<\alpha<2.
\end{align}
\end{definition}

Hence, $\zeta_{\alpha}(X_t)$ provides a measure of uncertainty for the residual lifetime random variable, expressed via the survival function $\bar{F}$, for different choices of $\alpha$.
Differentiating (\ref{ACRM t}) with respect to $t$, we get the following expression in terms of hazard rate $h(t)$ and mean residual life function $m(t)$.
\begin{align}\label{ACRM't}
(\alpha-1) \zeta'_\alpha(X_t)=\left((2-\alpha) \zeta_{\alpha}(X_t)-m(t)\right)h(t).
\end{align}

From the definition of DCRMHE of order $\alpha$, we can observe the following properties.
\begin{property}
For $t=0$, then $\zeta_\alpha(X_t)=\zeta_\alpha(X)$.
\end{property}
\begin{property}
 As $\alpha \to 1$,  $\zeta_\alpha(X_t)$ becomes the dynamic cumulative residual entropy function given in (\ref{CREt}).
 \end{property}

 \begin{property}
 Consider the random variable $Y=aX+b$ where $a>0$, $b\geq0$. Then we have
\[
\zeta_\alpha(Y;t)
= a\,\zeta_\alpha\!\left(X;\frac{t-b}{a}\right),
\qquad t\ge b.
\]
 \end{property}
    \begin{proof}
Let $Y=aX+b$, with $a>0$ and $b\geq 0$.
Then the survival function of $Y$ is given by
\[
\bar{F}_Y(y)=\bar{F}\!\left(\frac{y-b}{a}\right), \qquad y\ge b.
\]

The dynamic cumulative residual Mathai--Haubold entropy of order $\alpha$ of $Y$ at time $t$ is
\[
\zeta_\alpha(Y;t)
= \frac{1}{\alpha-1}
\left[
\int_t^\infty
\left(\frac{\bar{F}_Y(y)}{\bar{F}_Y(t)}\right)^{2-\alpha} \,dy
-
\int_t^\infty
\left(\frac{\bar{F}_Y(y)}{\bar{F}_Y(t)}\right) \,dy
\right].
\]

Substituting $\bar{F}_Y(y)=\bar{F}\!\left(\frac{y-b}{a}\right)$ and using the change of variable
$x=\frac{y-b}{a}$, we obtain
\[
\zeta_\alpha(Y;t)
= \frac{a}{\alpha-1}
\left[
\int_{\frac{t-b}{a}}^\infty
\left(\frac{\bar{F}(x)}{\bar{F}\!\left(\frac{t-b}{a}\right)}\right)^{2-\alpha} \,dx
-
\int_{\frac{t-b}{a}}^\infty
\left(\frac{\bar{F}(x)}{\bar{F}\!\left(\frac{t-b}{a}\right)}\right) \,dx
\right].
\]

Hence,
\[
\zeta_\alpha(Y;t)
= a\,\zeta_\alpha\!\left(X;\frac{t-b}{a}\right),
\qquad t\ge b.
\]
\end{proof}
\begin{corollary}
 From the above property, it is clear that\\
$(i)$ If $b=0$, then $\zeta_\alpha(aX;t)=a~\zeta_\alpha\left(X;\frac{t}
{a}\right)$.\\
$(ii)$ If $a=1$, then $\zeta_\alpha(X+b;t)=\zeta_\alpha(X;t-b)$.
 \end{corollary}
 \begin{property}
For the proportional hazards model with survival function
\[
\bar{F}_\theta^*(x) = \bigl(\bar{F}(x)\bigr)^\theta, \qquad \theta>0,\; x\in\mathbb{R},
\]
the dynamic cumulative residual Mathai--Haubold entropy of order $\alpha$ satisfies
\begin{equation}\label{aPH}
(\alpha-1)\,\zeta_{\alpha}(X_\theta^{*};t)
= (\beta-1)\,\zeta_{\beta}(X;t)
+ (\theta-1)\,\psi_{\theta}(X;t),
\end{equation}
where $\beta = 2-\theta(2-\alpha)$ and $\psi_{\theta}(X;t)$ denotes the dynamic cumulative
residual Tsallis entropy of order $\theta$ defined in (\ref{a eta t}).
\end{property}

\begin{corollary}
Consider a series system consisting of $n$ independent and identically distributed components with common lifetime distribution $F$. Let $X_{1:n}$ denote the system lifetime. Then the survival function of $X_{1:n}$ is
\[
\bar{F}_{X_{1:n}}(x)=\bigl(\bar{F}(x)\bigr)^n,
\]
and hence $X_{1:n}$ follows a proportional hazards model with parameter $\theta=n$.
Consequently, the cumulative residual Mathai--Haubold entropy of order $\alpha$ satisfies
\[
(\alpha-1)\,\zeta_{\alpha}(X_{1:n};t)
= (\beta-1)\,\zeta_{\beta}(X;t)
+ (n-1)\,\psi_{n}(X;t),
\]
where $\beta = 2-n(2-\alpha)$.
\end{corollary}

\begin{remark}
The above result shows that the dynamic cumulative residual Mathai--Haubold entropy of order $\alpha$ of a
series system can be expressed as a linear combination of the corresponding entropy
of a single component and the dynamic cumulative residual Tsallis entropy.
In contrast, for a parallel system with lifetime $X_{n:n}$ and survival function
$\bar{F}_{X_{n:n}}(x)=1-\{1-\bar{F}(x)\}^n$, the proportional hazards structure does not
hold in general, and therefore an identity of the form \eqref{aPH} is not directly
applicable.
\end{remark}

\begin{flushleft}
The next theorem provides the relationship of the dynamic cumulative residual Mathai--Haubold entropy of order $\alpha$, $\zeta_{\alpha}(X_t)$ with mean residual life function $m(x)$.
\end{flushleft}

\begin{theorem}
 Consider a non-negative random variable $X$ with density function $f(x)$, survival function $\bar{F}(x)$ and mean residual life function $m(x)$. Then
\[
\zeta_{\alpha}(X_t)
= \frac{{E}\!\left[ m(X)\, (\bar{F}(X))^{1-\alpha}\mid X>t \right]}
{(\bar{F}(t))^{1-\alpha}} .
\]
\end{theorem}
\begin{proof}
From(\ref{ACRM t}), we have
\begin{align}
(\alpha-1)\,\zeta_{\alpha}(X_t)
&= - \bigl(\bar{F}(t)\bigr)^{\alpha-2}
\int_{t}^{\infty}
\frac{d}{dx}\!\bigl(m(x)\,\bar{F}(x)\bigr)\,
\bigl(\bar{F}(x)\bigr)^{1-\alpha}\,dx
- m(t) \nonumber\\[6pt]
&= (\alpha-1)\,\frac{1}{\bigl(\bar{F}(t)\bigr)^{2-\alpha}}
\int_{t}^{\infty}
m(x)\,\bigl(\bar{F}(x)\bigr)^{1-\alpha}\,f(x)\,dx , \nonumber\\[6pt]
\zeta_{\alpha}(X_t)
&= \frac{1}{\bigl(\bar{F}(t)\bigr)^{1-\alpha}}
\int_{t}^{\infty}
m(x)\,\bigl(\bar{F}(x)\bigr)^{1-\alpha}\,\frac{f(x)}{\bar{F}(t)}\,dx .\label{relation}
\end{align}
Hence the proof.
\end{proof}
\begin{corollary}
    Consider a non negative random variable $X$ with mean residual life function $m(x)$. If $X$ has decreasing (increasing) mean residual life (DMRL) (IMRL), then
   \[
\zeta_{\alpha}(X_t) \leq (\geq) \frac{m(t)}{2-\alpha}.
\]
\end{corollary}
\begin{proof}
When $X$ has the DMRL (IMRL) property, we have
$m(x) \leq (\geq)\, m(t)$ for all $x \geq t$.
Using (\ref{relation}), the result follows.
\end{proof}
\begin{definition}
A random variable $X$ is said to have increasing (decreasing) dynamic cumulative residual Mathai--Haubold entropy
of order $\alpha$, denoted by IDCRMHE (DDCRMHE), if and only if
$\zeta_{\alpha}(X_t)$ is an increasing (decreasing) function of $t$.

\end{definition}
\begin{theorem}
The function $\zeta_{\alpha}(X_t)$ is increasing (decreasing) in $t$ if and only if
\[
\zeta_{\alpha}(X_t)\geq(\leq)\frac{m(t)}{2-\alpha}.
\]
\end{theorem}
\begin{proof}
   Differentiating \eqref{ACRM t} with respect to $t$, we obtain the following relationship
among the dynamic cumulative residual Mathai--Haubold entropy of order $\alpha$, the hazard rate
$h(t)$, and the mean residual life function $m(t)$ as
\begin{equation}\label{relationship}
\zeta_{\alpha}'(X_t)
=\bigl((2-\alpha)\zeta_{\alpha}(X_t)-m(t)\bigr)\,h(t).
\end{equation}
The desired result now follows directly from \eqref{relationship}.
\end{proof}
\begin{lemma}
If $X$ is smaller than $Y$ in the hazard rate order, denoted by $X \leq_{\mathrm{hr}} Y$,
then
\[
m_{F}(t) \leq m_{G}(t), \qquad t \ge 0,
\]
where $m_{F}(t)$ and $m_{G}(t)$ denote the mean residual life functions of $X$ and $Y$,
respectively.
\end{lemma}
\begin{proof}
The proof follows from the fact that if $X \leq_{\mathrm{hr}} Y$, then for all $x \ge t$,
\[
\frac{\bar{F}(x)}{\bar{F}(t)} \leq \frac{\bar{G}(x)}{\bar{G}(t)}.
\]
This inequality yields the desired ordering of the corresponding mean residual life
functions, and hence the proof is complete.
\end{proof}
The following theorem establishes that the dynamic cumulative residual Mathai--Haubold
entropy of order $\alpha$, $\zeta_{\alpha}(X_t)$, uniquely determines the survival function.

\begin{theorem}
 Consider a non-negative random variable $X$ with probability density function $f(x)$,
survival function $\bar{F}(x)$, hazard rate $h(x)$, and mean residual life function
$m(x)$. Suppose that $\zeta_{\alpha}(X_t)$ is increasing in $t$. Then
$\zeta_{\alpha}(X_t)$ uniquely determines the survival function $\bar{F}(x)$.
\end{theorem}
 \begin{proof}
Suppose that $F(x)$ and $G(x)$ are the distribution functions of the random variables
$X$ and $Y$, respectively, and assume that
\begin{equation}\label{A CRM F=G}
\zeta_{\alpha}(X_t)=\zeta_{\alpha}(Y_t), \qquad t \ge 0 .
\end{equation}
Differentiating both sides of \eqref{A CRM F=G} with respect to $t$ and using
\eqref{relationship}, we obtain
\begin{equation}\label{hazard}
\bigl((2-\alpha)\zeta_{\alpha}(X_t)-m_{1}(t)\bigr)\,h_{1}(t)
=
\bigl((2-\alpha)\zeta_{\alpha}(Y_t)-m_{2}(t)\bigr)\,h_{2}(t),
\end{equation}
where $h_{1}(t)$ and $h_{2}(t)$ denote the hazard rates, and $m_{1}(t)$ and $m_{2}(t)$
denote the mean residual life functions corresponding to $X$ and $Y$, respectively.

If $h_{1}(t)=h_{2}(t)$ for all $t \ge 0$, then $\bar{F}(t)=\bar{G}(t)$ for all
$t \ge 0$, and hence $F=G$. Therefore, it suffices to show that
$h_{1}(t)=h_{2}(t)$ for all $t \ge 0$.

Suppose, to the contrary, that there exists $t_0 \ge 0$ such that
$h_{1}(t_0)\neq h_{2}(t_0)$. Without loss of generality, assume that
$h_{1}(t_0)>h_{2}(t_0)$. Then, from \eqref{hazard}, we obtain
\[
(2-\alpha)\zeta_{\alpha}(X_{t_0})-m_{1}(t_0)
<
(2-\alpha)\zeta_{\alpha}(Y_{t_0})-m_{2}(t_0),
\]
which implies that
\[
m_{1}(t_0)>m_{2}(t_0).
\]
This contradicts Lemma~3.1, which states that $h_{1}(t_0)>h_{2}(t_0)$ implies
$m_{1}(t_0)<m_{2}(t_0)$. Hence, such a $t_0$ cannot exist, and therefore
$h_{1}(t)=h_{2}(t)$ for all $t \ge 0$. Consequently, $F=G$, which completes the proof.
\end{proof}


\section{Characterization results}
In this section, we  characterize some well known distributions based on dynamic cumulative residual Mathai--Haubold entropy of order $\alpha$, $\zeta_{\alpha}(X_t)$.\\
The following theorem shows that $\zeta_{\alpha}(X_t)$ is independent of $t$ if and only if $X$ follows exponential distribution.
\begin{theorem}
  Let $X$ be an absolutely continuous non-negative random variable with distribution function $F(x)$. Then $\zeta_{\alpha}(X_t)$ is independent of $t$ if and only if $X$ follows exponential distribution.
\end{theorem}
\begin{proof}
  Assume that $\zeta_{\alpha}(X_t)=k$, where $k$ is a positive constant. Then
\[
\zeta_{\alpha}'(X_t)=0.
\]
Using the relationship among $\zeta_{\alpha}(X_t)$, the hazard rate, and the mean
residual life function given in \eqref{relationship}, we obtain
\[
\bigl((2-\alpha)k - m(t)\bigr)\,h(t)=0.
\]
 it follows that
\[
m(t)=(2-\alpha)k,
\]
which is a constant. It is well known that a constant mean residual life function
characterizes the exponential distribution. Therefore, $X$ follows an
exponential distribution with parameter $\theta$.
\\  Conversely assume that $X\sim exp(\theta)$, then
  \[
  \zeta_{\alpha}(X_t) = \frac{1}{\theta(2-\alpha)}, \]
a constant. Hence it is clear that $\zeta_{\alpha}(X_t)$ is independent of $t$ if and only if $X$ has exponential distribution.  \end{proof}
  \begin{flushleft}
    The following theorem provides a characterization result in terms of the relationship connecting the $\zeta_{\alpha}(X_t)$ and the mean residual life function for the generalised Pareto distribution (GPD) .
\end{flushleft}
\begin{theorem}
    Let $X$ be a non-negative random variable with an absolutely continuous survival function $\bar{F}(x)$ and a mean residual life function $m(t)$, then the relationship
    \begin{align}\label{thm4.2}
     \zeta_{\alpha}(X_t)=k~m(t);
    \end{align}
    where $k=\frac{a+1}{(a+1)(2-\alpha)+a}$
    holds for every $t>0$, if and only if X follows GPD.
\end{theorem}
\begin{proof}
Suppose that \eqref{thm4.2} holds. Differentiating both sides with respect to $t$, we obtain
\[
\zeta_{\alpha}'(X_t)=k\,m'(t).
\]
From \eqref{relationship}, it follows that
\begin{equation}\label{proof4.2}
\bigl((2-\alpha)k-1\bigr)m(t)h(t)=k\,m'(t).
\end{equation}
Further, the hazard rate and the mean residual life function satisfy
\[
h(t)=\frac{1+m'(t)}{m(t)}.
\]
Substituting this expression for $h(t)$ into \eqref{proof4.2}, we obtain
\[
k\,m'(t)=\bigl((2-\alpha)k-1\bigr)\bigl(1+m'(t)\bigr),
\]
which yields
\[
m'(t)=\frac{(2-\alpha)k-1}{\,k-(2-\alpha)k+1\,},
\]
a constant. Hence, $m(t)$ is linear in $t$.

It is well known that a linear mean residual life function characterizes the
generalized Pareto distribution (GPD). Therefore, $X$ follows a GPD.

Conversely, assume that $X\sim\mathrm{GPD}$. By direct calculation, we obtain
\[
\zeta_{\alpha}(X_t)=k(b+at)=k\,m(t),
\]
where
\[
k=\frac{a+1}{(2-\alpha)(a+1)-a}.
\]
This completes the proof.
\end{proof}

\section{New classes of lifetime distributions}
 In this section, two new classes of lifetime distribution in terms of dynamic cumulative residual Mathai-Haubold entropy of order $\alpha$, $\zeta_{\alpha}(X_t)$ is given.\par

\begin{definition}
A random variable $X$ is said to have an increasing (decreasing) DCRMHE of order $\alpha$,  denoted by  IDCRMHE (DDCRMHE), if $\zeta_{\alpha}(X_t)$ increases (decreases) as $t$ increases for $t\geq 0$, which is same as $\zeta'_\alpha(X_t)\geq (\leq)0$.
\end{definition}

\begin{definition}
A random variable $X$ is said to have an increasing (decreasing) failure rate IFR (DFR), if $h(t)$ increases (decreases) as $t$ increases,  $t\geq 0$.
\end{definition}
\begin{flushleft}
 The following theorem provides bound in terms of $h(t)$ for $\zeta_\alpha(X_t)$.
\end{flushleft}
\begin{theorem}
    A random variable $X$ with distribution function $F$ possesses an increasing (decreasing) DCRMHE of order $\alpha$ if and only if for all $t>0$
  \[
    \]
    \[
    (i) \zeta_{\alpha}(X_t)\geq(\leq)\frac{m(t)h(t)}{2-\alpha}, for~ 1<\alpha<2
    \].
    \[
    (ii) \zeta_{\alpha}(X_t)\leq(\geq)\frac{m(t)h(t)}{2-\alpha}, for ~ 0<\alpha<1.
    \]

\begin{proof}
    The proof of the theorem directly follows from (\ref{relationship}) and Definition 5.1.
\end{proof}
\begin{flushleft}
\end{flushleft}
\end{theorem}
\begin{theorem}
    The uniform distribution over $(a,b),a<b$ can be distinguished by DDCRMHE of order $\alpha$ for $1<\alpha<2$ and IDCRMHE of order $\alpha$ for $0<\alpha<1$.
    \begin{proof}
        Let $X \sim U(a,b)$.\\ Then
        \[
        \bar{F}(x)=\frac{b-x}{b-a};a<x<b.
        \]
      \[
      \zeta_{\alpha}(X_t)=\frac{1}{\alpha-1}\left( {\int_t^b} \left(\frac{b-x}{b-t} \right)^{2-\alpha}\,dx - {\int_t^b} \left(\frac{b-x}{b-t}\right) \,dx \right)
      \]
      \[
   \quad   =\frac{1}{\alpha-1}\left(-\frac{(b-t)}{\alpha-3}-(b-t)\right).
      \]
      Differentiating with respect to $t$, we get
       \[
       \zeta'_{\alpha}(X_t)=\frac{\alpha-2}{(\alpha-1)(\alpha-3)}.
       \]
       It is clear that, when $0<\alpha<1$ \[\zeta'_{\alpha}(X_t)<0\]
       and when $ 1<\alpha<2$ \[\zeta'_{\alpha}(X_t)>0.
       \]
       Hence the theorem.
        \end{proof}
\end{theorem}
\begin{flushleft}
    The following theorem shows that the only distribution that is both IDCRMHE and DDCRMHE of order $\alpha$ is exponential.
\end{flushleft}
       \begin{theorem}
        Consider a random variable $X$ holding both IDCRMHE and DDCRMHE of order $\alpha$, then $X$ follows an exponential distribution.
        \begin{proof}
            If $X$ is DDCRMHE (IDCRMHE) of order $\alpha$, then  \[
             \zeta'_{\alpha}(X_t)\leq0(\geq0).
            \]
            From the above  inequalities, we get \[
            \zeta_{\alpha}(X_t)=k, \\\ a \\\ constant.
            \]

            By Theorem 4.1, constancy of $\zeta_{\alpha}(X_t)$ means that $X$ follows exponential distribution. This completes the proof.
        \end{proof}

       \end{theorem}



\section{Non-parametric Kernel estimation}
Let $X_1,X_2,...,X_n$ be a random sample taken from a population with distribution function
$F$. Based on kernel density estimation, we construct non-parametric estimators for the proposed measures and let us assume that the kernel function $k(x)$ satisfies the following conditions:
\begin{itemize}
    \item[1.] $k(x) \geq 0$, for all $x$
    \item[2.] $\int {k(x)} dx =1$
    \item[3.] $k(.)$ is symmetric.
\end{itemize}
 The probability density function $f(x)$ at a point $x$ can be estimated using the kernel density estimator given by (Parzen (1962))
\begin{equation}\label{f x}
    f_n(x)=\frac{1}{nh}\sum_{j=1}^{n}k\left(\frac{x-X_j}{h}\right),  h ~is ~the ~ bandwidth.
\end{equation}

Since our measure is defined based on survival functions, we choose a kernel based estimator for the survival function, which is given by
\begin{equation}\label{F x}
    \bar{F}(x)=\frac{1}{n}\sum_{j=1}^{n}\bar{K}\left(\frac{x-X_j}{h}\right),
\end{equation}
where $\bar{K}$ be the survival function of the kernel $k$ and $\bar{K}(t)=\int_{t}^{\infty}k(u)du$. \par

\sloppy
Non-parametrically, the CRMHE of order $\alpha$, $\zeta_{\alpha}(X)$ can be estimated using a kernel-based approach, defined as below

\begin{align}\label{CRMX hat}
    \widehat  {\zeta_{\alpha}}(X)=\frac{1}{\alpha-1}\left[\int_{0}^{\infty}\left(\frac{1}{n}\sum_{j=1}^{n}\bar{K}\left(\frac{x-X_j}{h}\right)\right)^{2-\alpha}dx-\int_{0}^{\infty}\left(\frac{1}{n}\sum_{j=1}^{n}\bar{K}\left(\frac{x-X_j}{h}\right)\right)dx\right].
\end{align}
Similarly, the estimator of DCRMHE of order $\alpha$, is as follows
\begin{align}\label{CRMt hat}
    \widehat  {\zeta_{\alpha}}(X_t)=\frac{1}{\alpha-1}\left[\int_{t}^{\infty}\left(\frac{\sum_{j=1}^{n}\bar{K}\left(\frac{x-X_j}{h}\right)}{\sum_{j=1}^{n}\bar{K}\left(\frac{t-X_j}{h}\right)}\right)^{2-\alpha}dx-\int_{t}^{\infty}\left(\frac{\sum_{j=1}^{n}\bar{K}\left(\frac{x-X_j}{h}\right)}{\sum_{j=1}^{n}\bar{K}\left(\frac{t-X_j}{h}\right)}\right)dx \right].
\end{align}
Next, we examined the consistency of the proposed estimators. The kernel based estimator of the cumulative distribution function $F(x)$, established by Berg and Politis (2009), is consistent and is defined as
\begin{equation}\label{F hat}
    \widehat{F}_h(x)=\int_{-\infty}^{t}\hat{f}(t)dx=\frac{1}{n}\sum_{j=1}^{n}\tilde{K}\left(\frac{t-X_j}{h}\right),
\end{equation}

where $\tilde{K}(t)=\int_{0}^{t}k(u)du$.\par

 They have also given the expression for the variance of $\hat{F}_h(t)$ to establish the consistency, which is given by
\begin{align}\label{var}
    Var(\widehat{F}_h(t))=\frac{F(t)(1-F(t))}{n}-\frac{2f(t)}{n}\left(\int u\tilde{K}(u)k(u)du\right)h+O\left(\frac{h}{n}\right).
\end{align}
Under certain assumptions, if $h\to 0$ and $nh \to \infty$ as $n \to \infty$, then $Var(\widehat{F}_h(t)) \to 0$ and hence the consistency of $\widehat{F}_h(t)$ is established. In order for the consistency of our proposed estimator to be proved, we need the following assumptions.\\
Let $\phi(t)$ denote the characteristic function of $X$:
\begin{itemize}
    \item[(A)] There is a $p >0$ such that $\int_{-\infty}^{\infty}|t|^{p}|\phi(t)|< \infty$.
    \item[(B)] There are positive constants $d$ and $D$ such that $|\phi(t)| \leq De^{-d|t|}$.
    \item[(C)] There is a positive constant $b$ such that $\phi(t)=0$ for all $|t| \geq b$.
\end{itemize}
    The consistency of the estimators has to be proved next. For this, first we prove the consistency of $\widehat{\bar{F}}(t)$. By direct calculation, it can be shown that
    \begin{equation*}
        \tilde{K}(t)=1-\int_t^\infty k(u)du=1-\bar{K}(t).
    \end{equation*}
    \begin{equation*}
    so \\\ \widehat{F}(t)=\frac{1}{n}\sum_{j=1}^{n}\left(1-\bar{K}\left(\frac{t-X_j}{h}\right)\right)=1-\hat{\bar{F}}(t).
    \end{equation*}
    Therefore, (\ref{var}) becomes
    \begin{align*}
     Var(\widehat{\bar{F}}_h(t))=\frac{\bar{F}(t)(1-\bar{F}(t))}{n}-\frac{2f(t)}{n}\left(\int u(1-\bar{K}(u))k(u)du\right)h+O\left(\frac{h}{n}\right)
    \end{align*}
    \begin{equation}\label{var final}
        That \\\ is,\\\\ \\\
        Var(\widehat{\bar{F}}_h(t))=\frac{\bar{F}(t)(1-\bar{F}(t))}{n}+\frac{2f(t)}{n}\left(\int u\bar{K}(u)k(u)du\right)h+O\left(\frac{h}{n}\right).
    \end{equation}
    It is clear from (\ref{var final}) that $\widehat{\bar{F}}_h(t)$ is a consistent estimator of $\bar{F}(t)$. Hence from (\ref{CRMX hat}), we can see that $\widehat  {\zeta_{\alpha}}(X)$ is a consistent estimator of $  {\zeta_{\alpha}}(X)$ and from (\ref{CRMt hat}), $\widehat  {\zeta_{\alpha}}(X_t)$ is a consistent estimator of $  {\zeta_{\alpha}}(X_t)$.\par

 \section {Simulation studies}
In this section, we study the Monte Carlo simulation of the estimators $ \widehat{\zeta}_{\alpha}(X)$ and $ \widehat{\zeta}_{\alpha}(X_t)$. The finite-sample performance of the proposed estimator can be evaluated by analyzing its bias and mean squared error under different sample sizes and distributions, which is the aim of the Monte Carlo simulation study. For performing this simulation, we use the R software and, for various sample sizes $n=30, 50, 70, 90$, the experiment is repeated 10,000 times. Here, the random variable $X$ is generated from two different lifetime random variables, namely, Weibull and uniform. For different values of $\alpha$, we randomly select the parameters and various sample sizes are used. The kernel survival estimator is used to estimate the proposed measure. Silverman's thumb rule is used to choose the bandwidth, and is given by
\[
h=1.06 \widehat\sigma {n}^ {-1/5},
\] where $\widehat\sigma$ be the standard deviation of $n$ samples taken. The estimates of $CRMHE$ and $DCRMHE$ of order $\alpha$ are calculated based on equations (\ref{CRMX hat}) and (\ref{CRMt hat}), and then the bias and MSE are also calculated.\par

 The bias and MSE of the CRMHE of order $\alpha$ estimator for different distributions are given in Table \ref{CRMHE alpha=1.5} for $\alpha=1.5$. Table \ref{CRMHE alpha=1.5} gives clear evidence that the estimator shows lower bias and MSE for uniform random samples compared to other distributions. Similarly, Table \ref{CRMHE alpha=0.5} gives the bias and MSE of the CRMHE of order $\alpha$ estimator for different distributions when $\alpha=0.5$. In Table \ref{CRMHE alpha=0.5}, we can see that the CRMHE of order $\alpha$ estimator performs better in the case of uniform samples compared to other distributions when $\alpha=0.5$. From Table \ref{CRMHE alpha=1.5} and Table \ref{CRMHE alpha=0.5}, it is clear that the bias and MSE decrease as $n$ increases.

\begin{table}[h!]
    \centering
    \caption{Bias and MSE of the CRMHE of order $\alpha$ estimator for various distributions when $\alpha=1.5$.}
    \label{CRMHE alpha=1.5}

\begin{tabular}{|c|cc|cc|}
\toprule
$n$ & \multicolumn{2}{c|}{$X \sim \text{Weibull}(5,1)$ } &  \multicolumn{2}{c|}{$X \sim \text{Uniform}(1.25,1.75)$}  \\

\cmidrule{1-5}

& Bias & MSE & Bias & MSE \\

\midrule
30 &0.0351& 0.0028& 0.0439& 0.0022\\

50& 0.0304& 0.0019 &0.0387& 0.0017\\

70& 0.0273& 0.0014&0.0355& 0.0014\\
90 &0.0252& 0.0011&0.0333& 0.0012 \\
\bottomrule
\end{tabular}
\end{table}

\begin{table}[h!]
\centering
\caption{Bias and MSE of the CRMHE of order $\alpha$ estimator for various distributions when $\alpha=0.5$.}
    \label{CRMHE alpha=0.5}

\begin{tabular}{|c|cc|cc|}
\toprule
$n$ & \multicolumn{2}{c|}{$X \sim \text{Weibull}(5,1)$ }  &  \multicolumn{2}{c|}{$X \sim \text{Uniform}(0.5,1)$}  \\
\cmidrule{1-5}

& Bias & MSE & Bias & MSE \\
\midrule
 30& 0.0168& 0.0007& 0.0117& 0.0002\\

 50& 0.0145& 0.0004&0.0101& 0.0002\\
70& 0.0131& 0.0003& 0.0092& 0.0001 \\
 90& 0.0120& 0.0003&0.0084& 0.0000 \\
 \bottomrule
\end{tabular}
\end{table}

Next, we study the performance of the DCRMHE of order $\alpha$ estimator for different values of $t$ and $n$. Table \ref{DCRMHE alpha=1.5} gives the bias and MSE of the DCRMHE of order $\alpha$ estimator when $X \sim \text{Weibull}(5,3)$  with $\alpha=1.5$. Similarly, Table \ref{DCRMHE alpha=0.5} gives the bias and MSE of the DCRMHE of order $\alpha$ estimator when $X \sim \text{Weibull}(5,1)$ with $\alpha=0.5$. It is evident from Tables \ref{DCRMHE alpha=1.5} and \ref{DCRMHE alpha=0.5} that the DCRMHE of order $\alpha$ estimator shows lower bias and MSE for weibull samples than for the other distributions. From Table \ref{DCRMHE alpha=1.5} and Table \ref{DCRMHE alpha=0.5}, it is clear that the bias and MSE decrease as the sample size $n$ increases.

\begin{table}[h!]
\centering
\caption{Bias and MSE of the DCRMHE of order $\alpha$ estimator across different $t$ and $n$ values for ${X \sim \text{Weibull}(5,3)}$ with $\alpha=1.5$}.
    \label{DCRMHE alpha=1.5}
\begin{tabular}{|c|c|cc|}
\toprule
$t$ & $n$ & \multicolumn{2}{c|}{$X \sim \text{Weibull}(5,3)$ } \\
\midrule
&  &Bias & MSE \\
\midrule
     &30 &0.1028 &0.0246\\
0.50      &50 &0.0895 &0.0164 \\
 &70 &0.0806 &0.0125 \\
     &90 &0.0745&0.0101 \\

     \midrule
     &30 &0.1001 &0.0235 \\
 0.75    &50 &0.0873 &0.0157 \\
 &70 &0.0788 &0.0120\\
     &90 &0.0728 &0.0098 \\

     \midrule
     &30 &0.0963 &0.0219 \\
     &50 &0.0842 &0.0147 \\
1.00 &70 &0.0760&0.0113 \\
     &90 &0.0704&0.0092\\

     \bottomrule
\end{tabular}
\end{table}

\begin{table}[h!]
\centering
\caption{Bias and MSE of the DCRMHE of order $\alpha$ estimator across different $t$ and $n$ values for ${X \sim \text{Weibull}(5,1)}$ with $\alpha=0.5$}.
    \label{DCRMHE alpha=0.5}
\begin{tabular}{|c|c|cc|}
\toprule
$t$ & $n$ & \multicolumn{2}{c|}{$X \sim \text{Weibull}(5,1)$ }  \\
\midrule
&  &Bias & MSE \\
\midrule
       &30 &0.0124 &0.0004   \\
0.50       &50 &0.0109 &0.0003   \\
 &70 & 0.0099 &0.0002  \\
       &90 &0.0091 & 0.0002  \\

       \midrule
       &30 &0.0113 & 0.0003  \\
       0.75&50 &0.0098 &0.0002   \\
        &70 &0.0088&0.0001 \\
       &90 &0.0082&0.0001   \\

       \midrule
       &30 &0.0121&0.0003   \\
  1.00     &50 &0.0105&0.0002   \\
 &70 &0.0095 &0.0002   \\
       &90 &0.0088 &0.0001   \\

        \bottomrule

\end{tabular}
\end{table}

\section{Data analysis}

This section deals with the two real-life datasets to analyze the proposed estimator given in (\ref{CRMt hat}). First, we consider the data consisting of  70 failure times of aircraft windshields measured in units of 1000 hours used by Helu et al. (2020).
    We estimate DCRMHE for the dataset with $\alpha=1.5$ and  evaluate its efficiency by comparing its estimated and theoretical values for different values of $t$. 
   The bias and mean squared error (MSE) of DCRMHE of order $\alpha$ are given in Table \ref{DCRMHEalpha=1.5} using the kernel estimator defined in equation (\ref{CRMt hat}), based on 10,000 bootstrap samples of size 70 for $\alpha=1.5$. From Table \ref{DCRMHEalpha=1.5}, it is clear that the proposed estimator and its theoretical values are closely matched, indicating that the accuracy of the estimation improves. This reveals that the level of uncertainty related to failure time is reduced for higher values of $t$. Further, we can see that compared to the true value, the bias is moderately small . From fig.\ref{fig:graph 1}, it is clear that the plot of $ \zeta_{\alpha}(X_t) $and $ \widehat{\zeta}_{\alpha}(X_t)$ against $t$ showing their decreasing trend for $\alpha=1.5$.

 \begin{table}[h!]
\centering
\caption{Bias and MSE of estimator for DCRMHE for different values of $t$ when $\alpha=1.5$}.
\label{DCRMHEalpha=1.5}
\begin{tabular}{c| c c c c}
\hline
$t$& \(\zeta_{\alpha}(X_t)\)& $ \widehat{\zeta}_{\alpha}(X_t)$& Bias & MSE \\
\hline
0.9 & 1.2925 & 1.0257&-0.0361&0.0302\\
1.0 & 1.2832&1.0120 &-0.0397& 0.0306\\

1.1 &  1.2716& 0.9964&-0.0429& 0.0310\\

1.2 &1.2579& 0.9788 &-0.0457 &0.0315\\

1.3 & 1.2424&0.9593 &-0.0483& 0.0321\\
\hline
\end{tabular}
\end{table}

\begin{figure}
    \centering
    \includegraphics[width=0.7\linewidth]{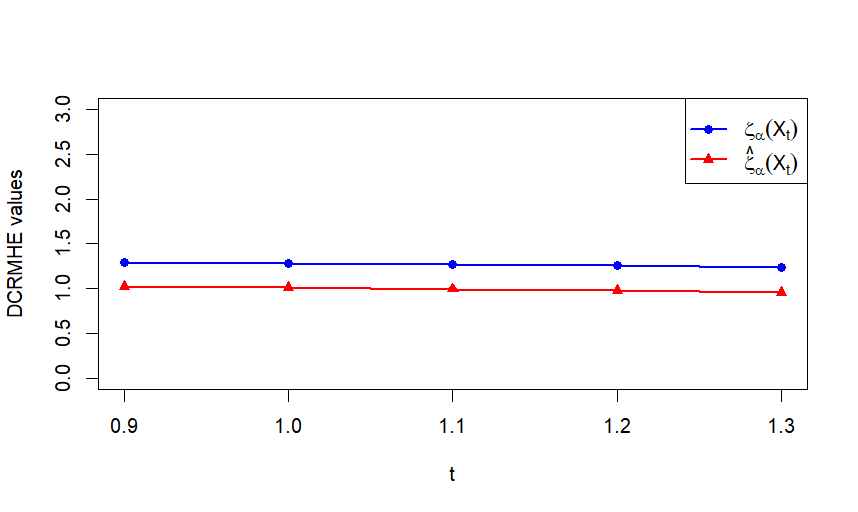}
    \caption{Comparison of $ \zeta_{\alpha}(X) $and $ \widehat{\zeta}_{\alpha}(X_t)$  values when $\alpha=1.5$ for different $t$}.
    \label{fig:graph 1}
\end{figure}
Next, we consider the failure times data (measured in millions of operations) of 40 randomly selected mechanical switches used by Nair (1984) . Here we fit a Weibull model to the data and the fit is checked using Kolmogrov-Smirnov (KS) test. We fit the Weibulll distribution on the shape parameter $k= 3.85819$ and the scale parameter $\lambda=2.3409$. For $\alpha=1.5$, the DCRMHE was estimated for the dataset. The bias and mean squared error (MSE) of DCRMHE are given in Table \ref{data 2} using the kernel estimator defined in equation (\ref{CRMt hat}), based on 10,000 bootstrap samples of size 40. From Table \ref{data 2}, it is clear that the proposed estimator and its theoretical values are closely matched, indicating that the accuracy of the estimation improves. From fig.\ref{fig:graph 2}, it is clear that the plot of $ \zeta_{\alpha}(X_t) $and $ \widehat{\zeta}_{\alpha}(X_t)$ against $t$ showing their decreasing trend for $\alpha=1.5$.

\begin{table}[h!]
\centering
\caption{Bias and MSE of estimator for DCRMHE for different values of $t$ when $\alpha=1.5$}.
\label{data 2}
\begin{tabular}{c| c c c c}\hline
$t$& \(\zeta_{\alpha}(X_t)\)& $ \widehat{\zeta}_{\alpha}(X_t)$ & Bias & MSE \\
\hline
0.9 &  0.8081 &0.9591&0.0624 &0.0129\\
1.0 & 0.7980&0.9495 &0.0605& 0.0124\\

1.1 &  0.7858& 0.9370&0.0588& 0.0120\\

1.2 & 0.7716&0.9218 &0.0573 &0.0116\\

1.3 &0.7554&0.9043 &0.0559& 0.0113\\
\hline
\end{tabular}
\end{table}

\begin{figure}
    \centering
    \includegraphics[width=0.7\linewidth]{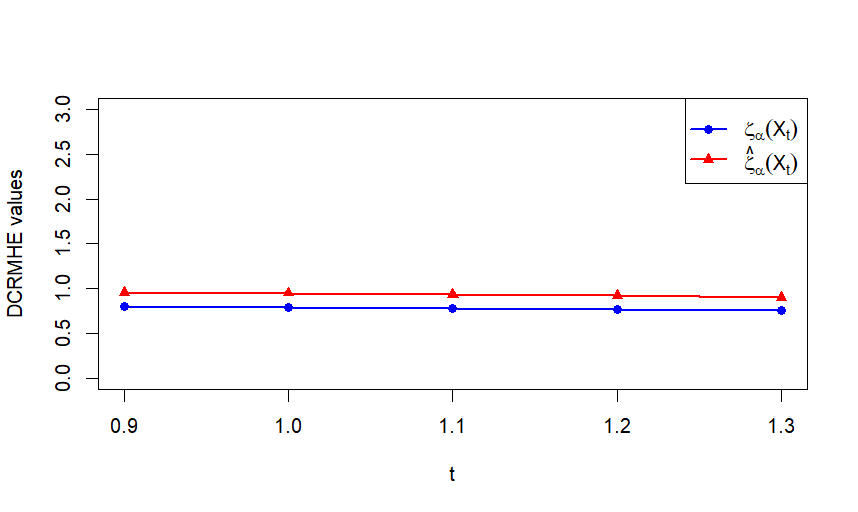}
    \caption{Comparison of $\alpha$ \(\zeta_{\alpha}(X)\) and $ \widehat{\zeta}_{\alpha}(X_t)$ values when $\alpha=1.5$ for different $t$}
    \label{fig:graph 2}
\end{figure}

\section{Conclusion}
In this paper, we studied the properties of the cumulative residual Mathai-Haubold entropy (CRMHE) and then propose a dynamic extension, DCRMHE, showing that it uniquely determines the distribution. Next, we studied the relationship connecting the hazard rate and the mean residual life function with DCRMHE and characterized some lifetime distributions. We also study the non-parametric estimators of CRMHE and DCRMHE based on the kernel density estimation of survival function, and their performance is assessed through a Monte Carlo simulation study. We used two real datasets on 70 failure times of aircraft windshields and failure times of 40 randomly selected mechanical switches, to know the relevance of the proposed DCRMHE estimator.
\section{References}
\begin{enumerate}

\item Alnssyan, B. S., \& Dar, J. G. (2026). Cumulative generalized entropies of Mathai--Haubold type and moment functions of order statistics. \textit{Mathematics}, 14, 3094.
\item Anija, C. R., Smitha, S., \& Kattumannil, S. K. (2025).
The cumulative residual Mathai--Haubold entropy and its nonparametric inference.
\item Asadi, M., \& Zohrevand, Y. (2007). On the dynamic cumulative residual entropy. \textit{Journal of Statistical Planning and Inference}, 137(6), 1931--1941.
\item Berg, A., \& Politis, D. N. (2009). Higher-order accurate bootstrap for studentized kernel density estimators. \textit{Journal of Statistical Planning and Inference}, 139(4), 1488--1496.
\item Cox, D. R. (1972). Regression models and life‐tables. \textit{Journal of the Royal Statistical Society}, 34, 187--202.
\item Dar, J. G. \& Al-Zahrani, B. (2013). On some characterization results of lifetime distributions using Mathai-Haubold residual entropy. \textit{IOSR Journal of Mathematics}, 5(4), 56--60.
\item Ebrahimi, N. (1996). How to measure uncertainty in  the residual lifetime distribution. \textit{Sankhya:The Indian Journal of Statistics, Series A}, 58(1), 48--56.
\item Golomb, S. (1966). The information generating function of probability distribution. \textit{IEEE Transactions on Information Theory} 12, 75-79.
\item Hall, W. J., \& Wellner, J. A. (1980). Mean residual life. \textit{In Statistics and related topics}, 169-–184. Ottawa, ON: North-Holland, Amsterdam.
\item Helu, A., Samawi, H., Rochaani, H., Yin, J., \& Vogel, R. (2020). Kernel density estimation based on progressive-II censoring. \textit{Journal of the Korean Statistical Society}, 49(2), 475--498.
\item    Mathai, A. M., \& Haubold, H. J. (2006). Pathway model, superstatistics, Tsallis statistics, and a generalized measure of entropy. \textit{Physica A:Statistical Mechanics and its Applications}, 375(1), 110--122.
\item Parzen, E. (1962). On estimation of a probability density function and mode. \textit{The Annals of Mathematical Statistics}, 33(3), 1065--1076.
\item Rao, M., Chen, Y., Vemuri, B.C., \& Wang, F. (2004). Cumulative residual entropy:A new measure of information. \textit{IEEE Transactions on Information Theory}, 50(6), 1220--1228.

\item Sati, M. M., \& Gupta, N. (2015). Some characterization results on dynamic cumulative residual Tsallis entropy. \textit{Journal of Probability and Statistics}, 1--8.
\item   Shannon, C. E. (1948). A mathematical theory of communication. \textit{Bell System Technical Journal}, 27(3), 379--423.
\item Smitha, S., Kattummannil, S. K., \& Sreedevi, E. P. (2023). Dynamic cumulative residual entropy generating function and its properties. \textit{Communications in Statistics - Theory and Methods}, 5890--5909.
\item Smitha, S., Rajesh, G., \& Jayalekshmi, S. (2024). On residual entropy generating function. \textit{Journal of the Indian Statistical Association}, 62(1), 81--93.%
\item Tsallis, C. (1988). Possible generalization of Boltzmann-Gibbs statistics. \textit{Journal of Statistical Physics}, 52,  479--487.

\end{enumerate}

\end{document}